\documentclass[12pt]{amsart}
\usepackage[letterpaper,left=2.5cm,top=2.5cm,bottom=2.75cm, 
right=2.5cm]{geometry}

\usepackage{amssymb,amsmath,mathrsfs,amscd, esint,amsthm}
\usepackage{graphicx, xcolor,amsfonts,times,enumerate}

\newcommand{\R}{\mathbb{R}}
\newcommand{\Rn}{{\R^n}}

\newcommand{\N}{\mathbb{N}}

\newcommand{\eps}{\varepsilon}
\newcommand{\Chi}{{\mathcal{X}}}

\newcommand{\cS}{{\mathscr{S}}}

\newcommand{\cQ}{{\mathcal{Q}}}
\newcommand{\cB}{{\mathcal{B}}}

\newcommand{\cR}{{\mathcal{R}}}
\newcommand{\cI}{{\mathcal{I}}}
\newcommand{\cO}{{\mathcal{O}}}

\newcommand{\ftilde}{{\widetilde{f}}}
\newcommand{\fast}{f^\ast}

\newcommand{\Lone}{{L^1}}
\newcommand{\Loneloc}{{L^1_{\text{loc}}}}

\newcommand{\ra}{\rightarrow}

\newcommand{\bmo}{{\textnormal{BMO}}} 

\def\BMO#1#2{\textnormal{BMO}_{#1}^{#2}}

\newcommand{\vmo}{{\textnormal{VMO}}}
\def\VMO#1#2{\textnormal{VMO}_{#1}^{#2}}

\newcommand{\cmo}{{\textnormal{CMO}}}

\def\XXint#1#2#3{{\setbox0=\hbox{$#1{#2#3}{\int}$ }
		\vcenter{\hbox{$#2#3$ }}\kern-.57\wd0}}

\newtheorem{introtheorem}{Theorem}

\newtheorem{theorem}{Theorem}[section]
\newtheorem{lemma}[theorem]{Lemma}

\newtheorem{corollary}[theorem]{Corollary}
\newtheorem*{corollary*}{Corollary}

\theoremstyle{definition}
\newtheorem{definition}[theorem]{Definition}
\newtheorem{example}[theorem]{Example}

\theoremstyle{remark}
\newtheorem{remark}[theorem]{Remark}
\newtheorem{innercustomthm}{Property}
\newenvironment{property}[1]
{\renewcommand\theinnercustomthm{{\em #1}}\innercustomthm} 
{\endinnercustomthm}

\numberwithin{equation}{section}

\begin{document}
	
	
	\title[Vanishing mean oscillation and continuity of rearrangements]
{Vanishing mean oscillation and continuity of rearrangements}
	
	\author[Burchard]{Almut Burchard}
	\address{(A.B.) University of Toronto, Department of Mathematics, Toronto, ON M5S 2E4, Canada}
	\curraddr{}
	\email{almut@math.toronto.edu}
	
	\author[Dafni]{Galia Dafni}
	\address{(G.D.) Concordia University, Department of Mathematics and Statistics, Montr\'{e}al, QC H3G 1M8, Canada}
	\curraddr{}
	\email{galia.dafni@concordia.ca}
	\thanks{A.B. was partially supported by Natural Sciences and Engineering Research Council (NSERC) of Canada. G.D. was partially supported by the Natural Sciences and Engineering Research Council (NSERC) of Canada and the Centre de recherches math\'{e}matiques (CRM). R.G. was partially supported by the Centre de recherches math\'{e}matiques (CRM), the Institut des sciences math\'{e}matiques (ISM), and the Fonds de recherche du Qu\'{e}bec -- Nature et technologies (FRQNT)}
	
	\author[Gibara]{Ryan Gibara}
	\address{(R.G.) University of Cincinnati, Department of Mathematical Sciences, Cincinnati, OH 45221-0025, USA}
	\curraddr{}
	\email{ryan.gibara@gmail.com}
	
	
	\date{January 13,2022}
	
	\begin{abstract}
	We study the decreasing rearrangement of functions in VMO, and
show that for rearrangeable functions, the mapping $f \mapsto \fast$ preserves
vanishing mean oscillation.  Moreover, as a map on BMO, while bounded, it
 is not continuous, but continuity holds at points in VMO (under certain
conditions). This also applies to the symmetric decreasing rearrangement.
Many examples are included to illustrate the results.
		
	\end{abstract}
	
	\maketitle
	
	
	\section{{\bf  Introduction}}
	\label{sec:intro}
	
	
	For equimeasurable rearrangements, boundedness and continuity do not always go hand-in-hand.  On Lebesgue spaces, both the decreasing rearrangement and the symmetric decreasing rearrangement are non-expansive, but the situation is more complicated for the Sobolev spaces $W^{1,p}$, $1\le p<\infty$. The P\'olya--Szeg\H{o} inequality $\|\nabla Sf\|_{p}\le \|\nabla f\|_p$ ensures that symmetrization decreases the norm in these spaces. Furthermore, Coron~\cite{cor} proved that this rearrangement is continuous on $W^{1,p}(\R)$. However, Almgren and Lieb~\cite{al} discovered that continuity fails in $W^{1,p}(\Rn)$ in all higher dimensions $n>1$. In essence, convergence can fail because the symmetric decreasing rearrangement can reduce the measure of the critical set where $\nabla f$ vanishes. In one dimension, this is precluded by Sard's lemma; for the same reason, Steiner symmetrization is continuous in any dimension~\cite{bur}. Similar questions have been studied regarding the boundedness and continuity of maximal functions on Sobolev spaces and BV \cite{apl, cmp, luiro, mad}.
	
	Turning to the space BMO of functions of bounded mean oscillation, it is well known that the decreasing rearrangement (that is, the map $f\mapsto f^*$) is bounded on BMO: there are constants $C_n$, depending only on dimension, such that 
	$$
	\|\fast\|_{\bmo}\leq C_n \|f\|_{\bmo}
	$$
	whenever the decreasing rearrangement is defined. The sharp dependence of the constants $C_n$ or, indeed, if they do depend on dimension at all is still an open question for dimensions $n>1$. See \cite{BDG11} for a discussion of this inequality and for a proof that $C_n$ exhibits at most square-root-dependence on $n$. It was proven in \cite{BDG11} that the symmetric decreasing rearrangement (that is, the map $f\mapsto Sf$) is also bounded on BMO. The sharp constants and their dependence on dimension remains unknown for $n\geq{1}$. 
	
	In this paper, we address the question of continuity of these rearrangements on BMO. This question has not previously been studied in the literature, even in dimension $n=1$. We show by means of an example (see Example~\ref{ex-discont}) that the decreasing rearrangement is discontinuous on $\bmo(\R^n)$ and it follows that the same is true for the symmetric decreasing rearrangement. 
	The phenomenon in Example~\ref{ex-discont} is somewhat similar to that described above for the Sobolev spaces: the sequence $f_k$ and the limit $f=f^*$ have jumps of height 1, but the decreasing rearrangement erases the jumps, and $f_k^*$ is continuous. Unlike the situation in $W^{1,p}$, this can happen even in one dimension.

	To eliminate the possibility of jump discontinuities, we consider the subspace VMO of functions of vanishing mean oscillation, which often plays the role of the continuous functions within BMO. The definition of VMO
	originates with Sarason \cite{sar}, who identified the closure of the uniformly continuous functions in BMO with those functions whose mean oscillation over any cube converges to zero uniformly in the diameter of the cube. VMO can be viewed as the $0$-endpoint on the smoothness scale: vanishing mean oscillation is a common 
	minimal regularity condition on the coefficients of PDE~\cite{cfl, kry} and on the normal to the boundary of non-smooth domains~\cite{mms}. 
	
	Our first result shows that the decreasing rearrangement $f^*$ of a function $f\in\vmo$ is in $\vmo$. As a decreasing function of a single variable, $f^*$ has vanishing mean oscillation if and only if it satisfies a sub-logarithmic growth condition (i.e., a vanishing John-Nirenberg inequality) at the origin. Note that for functions in the critical Sobolev spaces $W^{s,n/s}(\Rn)$, which embed in VMO, stronger vanishing at the origin was proved by Hansson \cite{han} and Brezis-Wainger \cite{bw}. 
	
	\begin{introtheorem}[Boundedness]\label{introbound}
		Let $Q_0\subset\R^n$ be a cube. If $f\in \vmo(Q_0)$, then $f^\ast\in\vmo(0,|Q_0|)$.
	\end{introtheorem}
	
	\noindent It turns out that functions in VMO on an unbounded domain are not automatically rearrangeable (see Example~\ref{ex-trees}). Nonetheless, on $\R^n$, we have the following anaologue of the above theorem:
	$$
	f\in \vmo(\R^n)\;\text{rearrangeable}\;\implies f^\ast\in\vmo(\R_+).
	$$
	See Theorem \ref{thm-vmobound} for this result and in a more general context. 
	
	Our second result shows that the decreasing rearrangement is continuous at all points in VMO, in the following sense. 
	
	\begin{introtheorem}[Continuity]
\label{introcont}
		Let $Q_0\subset\R^n$ be a cube. If $f_k$, $k\in\N$, are in $\bmo(Q_0)$ and $f$ is in $\vmo(Q_0)$ with $f_k \to f$ in $\bmo(Q_0)$ such that the means $\fint_{Q_0}f_k$ converge to $\fint_{Q_0}f$, then $f_k^*\to f^*$ in $\bmo(0,|Q_0|)$.
	\end{introtheorem}
	
	\noindent Similar continuity results hold on $\R^n$ and in more general settings -- see Theorem \ref{thm-vmo-infinite}. The corresponding conclusions (see Corollaries \ref{corsymbound} and \ref{corsymcont}) also hold for the symmetric decreasing 
	rearrangement.
	
	
	
	Both theorems will be proved more generally for functions on a domain $\Omega\subset\Rn$ that have vanishing mean oscillation with respect 
to a suitable basis $\cS$ in $\Omega$.  The VMO space in this 
setting is introduced in Section \ref{vmosection}.
	
	
	
	\section{{\bf Preliminaries}}
	\label{sec:prelim}
	
	
	\subsection{Rearrangements}
	To define the decreasing rearrangement (also called  the nonincreasing rearrangement) and the symmetric decreasing rearrangement, we will restrict ourselves to functions which we call {\em rearrangeable}.  For
	a measurable function $f$ on a domain $\Omega\subset\R^n$, this means that $\mu_f(\alpha)\rightarrow{0}$ as $\alpha\rightarrow\infty$, where $\mu_f$ is the distribution function of $f$. Recall that for $\alpha\geq{0}$, $\mu_f(\alpha) = |E_\alpha(f)|$, the Lebesgue measure of the level set
	$$E_\alpha(f):=\{x\in \Omega:|f(x)|>\alpha \}\,.
	$$
	
	The {\em decreasing rearrangement} is the right-continuous generalized inverse of the distribution function, given by
	\begin{equation}
		\label{eq-fast-alt}
		f^*(s)= \mu_{(\mu_f)}(s)=|\{\alpha\ge 0: \mu_f(\alpha)>s\}|\,,
	\end{equation}
see Figure~\ref{fig-fstar}.
	If the domain has finite measure $|\Omega|$, then $\fast(s)=0$ for all $s\geq|\Omega|$, so we consider $\fast$ as a function on $(0,|\Omega|)$.
	
	
	
	
	\begin{figure}[t] \label{fig-fstar}
	\includegraphics[width=0.8 \linewidth]{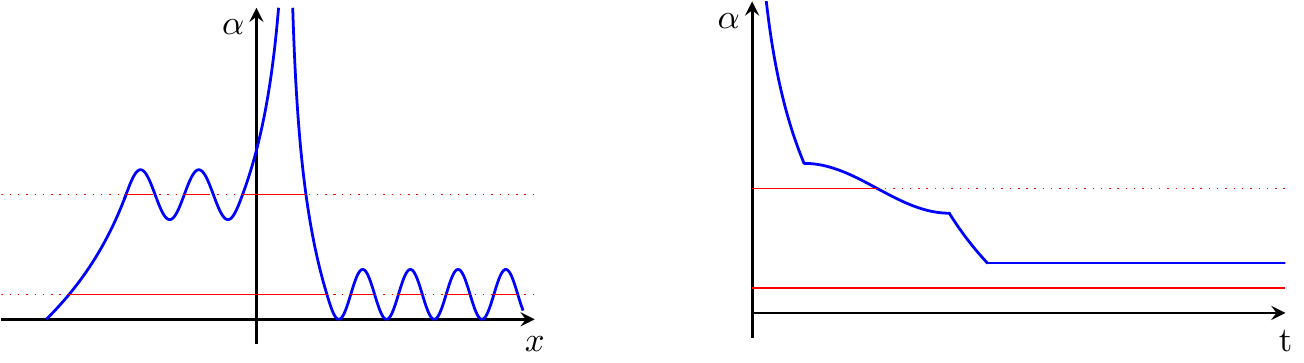}
	\caption{\small A nonnegative function $f\in \bmo(\R)$ 
and its decreasing rearrangement $\fast\in\bmo(\R_+)$.
}
	\end{figure}
	
	The following standard properties of the decreasing rearrangement will be used throughout this paper.
	
	\begin{property}{R1}
		\label{R-equi} {\em (Equimeasurability.)}\ 
		For all $\alpha\geq{0}$, $\mu_f(\alpha)=\mu_{\fast}(\alpha)$. 
	\end{property}
	
This property uniquely characterizes $f^*$ among the right-continuous decreasing functions on $(0,|\Omega|)$.
	
	
	\begin{property}{R2}
		\label{R-Lp}  The decreasing rearrangement $f\mapsto{\fast}$ is 
		norm-preserving from $L^p(\Omega)$ to $L^p(0,|\Omega|)$ for 
		all $1\leq{p}\leq\infty$. Furthermore, it is non-expansive and, 
		therefore, continuous.
	\end{property}
	
	\begin{property}{R3}
		\label{R-ptwise} If $f_k$, $k\in \N$, and $f$ are nonnegative rearrangeable functions on $\Omega$ satisfying $f_k\uparrow f$ pointwise, then $\fast_k\uparrow \fast$ pointwise on $(0,|\Omega|)$.
	\end{property}
	
	To see this, note first that for any $\alpha>0$, the level sets of $f$ satisfy $E_\alpha(f)=\bigcup_{k\ge 1}E_\alpha(f_k)$. By continuity of the measure, $\mu_{f_k}(\alpha)\uparrow \mu_f(\alpha)$ for all $\alpha>0$.  
	By Eq.~\eqref{eq-fast-alt}, $\fast$ and $\fast_k$ can be represented as the distribution functions of $\mu_f$ and $\mu_{f_k}$, respectively.
	Therefore, by the same argument as above, $\fast_k(s)\uparrow \fast(s)$ for all $s>0$.
	
	
	\begin{property}{R4}
		\label{R-trunc} {\em (Truncation.)}\ 
For any nonnegative rearrangeable function $f$ and any 
$0\leq \alpha<\beta\leq\infty$,
		$$
		\left(\min\{\max\{f,\alpha\},\beta\}\right)^*= \min\{\max\{\fast,\alpha\},\beta\}\,.
		$$
	\end{property}
	
	For a rearrangeable function $f$ on $\R^n$, we define its {\em symmetric decreasing rearrangement} $Sf$ by
	\begin{equation}\label{symmrea}
		Sf(x)=\fast(\omega_n|x|^n)\,,\qquad x\in\R^n\setminus \{0\}\,,
	\end{equation}
	where $\omega_n$ is the volume of the unit ball in $\Rn$.
	The symmetric decreasing rearrangement, as a map from functions on $\Rn$ to functions on $\Rn$, inherits Properties~{\ref{R-equi}}-{\ref{R-trunc}}.

	For later use, we record the behaviour of distribution functions and rearrangements under scaling, dilation, and translation. If $\widetilde f(x)= a f(b^{-1}(x-x_0))$ for some positive constants $a,  b$ and some $x_0\in\R^n$, then
	\begin{equation}
		\label{eq-trans-dil}
		\mu_{\widetilde f}(\alpha) = \bigl|\bigl\{x\in\Omega: |f(b^{-1}(x-x_0))|>a^{-1}\alpha\bigr\}\bigr|
		= b^n\mu_{f}(a^{-1}\alpha)\,.
	\end{equation}
	Since $f$, $f^*$, and $Sf$ are equimeasurable, by reversing the above calculation, we get 
	$$
	(\widetilde f)^*(s) =a f^*(b^{-n}s)\quad \text{and}\quad S \widetilde f(x) =a (Sf)(b^{-1}x)\,.
	$$
	
	More details  on the decreasing rearrangement can be found in \cite{sw}; see \cite{Baernstein}
	for the symmetric decreasing rearrangement.
	
	
	\subsection{Mean oscillation}
	
	John and Nirenberg~\cite{jn} introduced  functions of bounded mean oscillation over cubes in $\Rn$ with sides parallel to the axes.   We follow the terminology used in \cite{dg} in order to define mean oscillation over more general sets than cubes or balls. A {\em basis} of shapes in a domain $\Omega\subset\R^n$ is a collection $\cS$ of open sets $S\subset\Omega$,  $0<|S|<\infty$, forming a cover of $\Omega$. We will use $\cI$ to denote the basis of finite open intervals in $\R$; in $\Rn$, we denote by $\cB$ the basis of Euclidean balls, by $\cQ$ the basis of cubes with sides parallel to the axes, and by $\cR$ the rectangles with sides parallel to the axes.
	
	Let $f$ be a real-valued function with $f \in L^1(S)$ for every $S \in \cS$.  Define the {\em mean oscillation} of $f$ on a shape $S\in\cS$ by
	$$
	\cO(f,S):=\fint_{S}\!|f-f_S|\,,
	$$
	where $|S|$ denotes the measure of $S$ and $f_S := \fint_S f$ is the average of $f$ over $S$. 
	
	We will use the following properties of mean oscillation of an integrable function over a given shape $S \in \cS$. These {\em shapewise} identities and inequalities are proved in~\cite{dg}.
	
	\begin{property}{O1}
		\label{O-const}
		For any constant $\alpha$, 
		$\cO(f+\alpha,S)=\cO(f,S)$.
	\end{property}
	
	\begin{property}{O2}
		\label{O-alt} 
		Denoting $y_+=\max(y,0)$, 
		$$
		\cO(f,S)=\frac{2}{|S|}\int_{S}(f-f_S)_+\, .
		$$
	\end{property}
	\noindent When $f=\Chi_{E}$ for a measurable set $E$, then 
	\begin{equation} \label{eq-char-func}
		\cO(\Chi_{E}, S) = 2\rho(E,S)(1-\rho(E,S))\,,\;
\text{where}\; \rho(E,S):=\frac{|E \cap S|}{|S|}\,.
	\end{equation}
	
	\begin{property}{O3} 
		\label{O-abs}
		$\cO(|f|,S)\leq 2\cO(f,S)$. 
	\end{property}
	
	\begin{property}{O4}
		\label{O-median} 
		$$
		\cO(f,S) \leq 2 \inf_{\alpha}\fint_{S}\!|f-\alpha|=2 \fint_{S}\!|f-m|\,,
		$$
		where the infimum is taken over all constants $\alpha$, and $m$ is a median of $f$ on $S$,
		defined by the property that $|\{x\in S:f(x)>m\}|\leq\frac{1}{2}|S|$ and $|\{x\in S:f(x)<m\}|\leq\frac{1}{2}|S|$.
	\end{property}
	
	\begin{property}{O5} \label{O-2int}
		$$
		\cO(f,S) \leq \fint_S\fint_S \!|f(x) - f(y)| dx dy \leq 2\cO(f,S)\,.
		$$
	\end{property}
	
	\begin{property}{O6} \label{O-trunc} 
		For 
		$-\infty\leq \alpha<\beta\leq\infty$, the truncation $\ftilde=\min\{\max\{f,\alpha\},\beta\}$ satisfies
		$$
		\cO(\ftilde,S)\leq\cO(f,S)\,.
		$$
	\end{property}
	
	We will also frequently use the following comparison principle, which can be obtained by applying Property~{\ref{O-alt}} to both sides.
	
	\begin{property}{O7}
		\label{O-subset} For any  pair of shapes $S\subset\widetilde S$,
		$$
		\cO(f,S)\leq \frac{|\widetilde{S}|}{|S|}\, \cO(f,\widetilde S)\,.
		$$
	\end{property}
	
	If $\cS$ and $\widetilde{\cS}$ are two bases of shapes in $\Omega$ such that for every $S\in\cS$ there exists $\widetilde{S}\in\widetilde{\cS}$ with $S\subset\widetilde{S}$ and $|\widetilde{S}|\leq c|S|$, and for every $\widetilde{S}\in\widetilde{\cS}$ there exists $S\in\cS$ with $\widetilde{S}\subset S$ and $|S|\leq \widetilde{c}|\widetilde{S}|$, for some constants $c,\widetilde{c}>0$, then we say that
	$\cS$ is {\em equivalent} to $\widetilde{\cS}$, written $\cS\approx\widetilde{\cS}$.
	%
	
	
	
	\begin{definition}
		A function $f$ has {\em bounded mean oscillation} with respect to a basis $\cS$, denoted $f \in \BMO{\cS}{}(\Omega)$, if $f\in L^1(S)$ for all $S\in\cS$ and
		\begin{equation}
			\label{eq-bmo}
			\|f\|_{\BMO{\cS}{}}:=\sup_{S\in\cS} \cO(f,S)<\infty\, .
		\end{equation}
	\end{definition}
	When $\cS=\cQ$, the basis of cubes, $\BMO{\cS}{}(\Omega)$  will be simply denoted by $\bmo(\Omega)$. Many interesting properties of $\BMO{\cS}{}$ for $\cS=\cQ$ and $\cR$ can be found in ~\cite{ko2}.

	\begin{remark}\label{integrability}
	Functions in $\BMO{\cS}{}(\Omega)$ are locally integrable in $\Omega$. Note, however, that shapes need not be compactly contained in $\Omega$. In particular, functions in $\BMO{}{}(\R_+)$ are also integrable at the origin in the sense of being integrable on $(0,b)$ for any $b<\infty$.
	\end{remark}
	
	Property~{\ref{O-const}} tells us that Eq.~\eqref{eq-bmo} defines a seminorm that vanishes on constant functions. It is therefore natural to consider $\BMO{\cS}{}(\Omega)$ modulo constants, and it was shown in \cite{dg} that this gives a Banach space. When dealing with rearrangements, however, working modulo constants will not serve our purpose, and we will just consider $\BMO{\cS}{}(\Omega)$ as a linear space with a seminorm. Convergence in $\BMO{\cS}{}(\Omega)$ will always mean convergence with respect to this seminorm; i.e. 
	$$
	f_k\rightarrow f\ \text{in}\ \BMO{\cS}{}(\Omega)\quad\Longleftrightarrow\quad \lim \|f_k-f\|_{\BMO{\cS}{}}=0\,.
	$$
	
	We collect here some properties of BMO that will be used subsequently.
	
	\begin{property}{B1}
		\label{B-equiv} If $\cS\approx\widetilde{\cS}$, then $f\in\BMO{\cS}{}(\R^n)$ if and only if $f\in\BMO{\widetilde{\cS}}{}(\R^n)$, with
		$$
		\widetilde{c}^{-1} \|f\|_{\BMO{\widetilde{\cS}}{}}
		\le \|f\|_{\BMO{\cS}{}}\le {c}
		\|f\|_{\BMO{\widetilde{\cS}}{}}\,.
		$$ 
		This follows from Property~{\ref{O-subset}}.
	\end{property}
	
	\begin{property}{B2} \label{B-L1} 
		For any shape $S \in \cS$, $\|f - f_S\|_{L^1(S)}\leq |S| \|f\|_{\BMO{\cS}{}}.$ 
	\end{property}
	
	\begin{property}{B3} \label{B-trans-dil} 
		On $\Omega=\R^n$, if $\widetilde f(x)= f(b^{-1}(x-x_0))$ for some $b>0$ and $x_0\in\R^n$, then 
		$$
		\|\widetilde f\|_{\BMO{}{}}= \|f\|_{\BMO{}{}}\,.
		$$
	\end{property}
	
	
	
	\begin{property}{B4} \label{B-support} 
		Suppose $f$ is nonnegative and $S \in \cS$. Then
		$$
			\cO(f,S)\leq 2\left(\inf\frac{|S'|}{|S|}\right) \|f\|_{\BMO{\cS}{}}\,,
		$$
		where the infimum is taken over all 
shapes $S'$ with $S' \supset E_0(f) \cap S$ and 
$|S'|\geq 2 |E_0(f)\cap S'|$. 
On such a shape $S'$, we have that $m = 0$ is a median of $f$, and by Properties~{\ref{O-alt}} and~{\ref{O-median}},
		\begin{equation}\label{eq-B-support2}
		\cO(f,S) \leq \frac{2}{|S|} \int_{S}\!f
		\leq  \frac{2}{|S|} \int_{S'}\!|f-m| 
		= \frac{2|S'|}{|S|} \fint_{S'}\!|f-m|
		\leq \frac{2|S'|}{|S|} \cO(f,S')\,.
		\end{equation}
	\end{property} 
	
	In general, this is not sharp. For intervals on $\R$, if $E_0(f)$ is an interval $I$ then for any interval $S$ with $|I \cap S| < |S|$, the minimum is attained by taking an interval $S'$ of length $2|I\cap S|$, and we get the estimate
	\begin{equation}\label{eq-B-support3}
	\cO(f,S)\leq \frac{4|I \cap S|}{|S|} \|f\|_{\BMO{}{}}\,.
	\end{equation}
	This is sharp: for indicator functions, the norm of $\|\Chi_{E}\|_{\bmo(\R)}=\frac{1}{2}$ and $\rho(E,S)$ can be taken arbitrarily close to zero in Eq. \eqref{eq-char-func}.

	
	\subsection{Vanishing mean oscillation}\label{vmosection}
	
	An important subspace of BMO is the space of functions of vanishing mean oscillation, VMO, originally defined by Sarason on $\R$. We generalize the definition here to a basis $\cS$ of shapes in a domain $\Omega\subset\Rn$.
	\begin{definition} \label{vmodef}
		We say that a function $f\in\BMO{\cS}{}(\Omega)$ is in $\VMO{\cS}{}(\Omega)$ if
		\begin{equation}\label{vmo}
			\lim_{\delta\rightarrow{0}^+}\sup_{|S|\leq \delta}\cO(f,S)=0\,,
		\end{equation}
		where the supremum is taken over all shapes $S\in\cS$ of measure at most $\delta$. 
	\end{definition}
	
	In what follows, we want to exclude the possibility that Eq.~\eqref{vmo} holds vacuously because there are no shapes of arbitrarily small measure (an example is the basis of cubes with sidelength bounded below by some constant). This is implicit in the density condition~\eqref{eq-meas-cont}.
	
	By Properties~{\ref{O-subset}} and~{\ref{B-equiv}}, equivalent bases define the same VMO-space, with equivalent BMO-seminorms. Again, the notation $\VMO{}{}(\Omega)$ will be reserved for the case $\cS=\cQ$. For bases equivalent to cubes, having vanishing diameter is the same as having vanishing measure, hence the supremum in Definition~\ref{vmodef} can instead be taken over all shapes $S$ of diameter at most $\delta$. For general bases, however, vanishing diameter is strictly stronger than vanishing measure. Consider the basis of rectangles, $\cR$, in which a sequence of rectangles of constant diameter can have measure tending to zero. 
	
	
	For nice domains $\Omega$, $\vmo(\Omega)$ is the closure, in the $\bmo$-seminorm, of the set of uniformly continuous functions in $\bmo(\Omega)$ (see \cite{bd}). General functions in $\vmo$, though, need be neither continuous nor bounded; an example is $(-\log |x|)_+^p$ for $0 < p < 1$. On $\Omega  = \Rn$, $\vmo$ can also be characterized as the subset of $\bmo(\Rn)$ on which translation is continuous. In the case when $\Omega$ is unbounded, note that there is a strictly smaller $\vmo$-space, sometimes denoted $\cmo$ (see~\cite{bo,bd,uch}), in which additional vanishing mean oscillation conditions are required as the cube or its sidelength go to infinity.
	
	For any choice of basis, we have that $\VMO{\cS}{}(\Omega)$ is a closed 
	subspace of $\BMO{\cS}{}(\Omega)$. To see this, let $f_k \ra f$ in $\BMO{\cS}{}(\Omega)$, where $f_k \in \VMO{\cS}{}(\Omega)$ for all $k$. Then $f \in \VMO{\cS}{}(\Omega)$ since
	$$
	\sup_{|S|\leq \delta}\cO(f,S) \leq \sup_{|S|\leq \delta}\cO(f_k,S) 
	+ \|f - f_k\|_{\BMO{\cS}{}(\Omega)}
	$$
	can be made arbitrarily small by choosing $k$ sufficiently large and $\delta$ sufficiently small.
	

	
	\subsection{Rearrangement meets mean oscillation}
	
	To deal with the decreasing rearrangement for functions in BMO, we need to establish several conventions.  First, as noted above, while mean oscillation is invariant under the addition of constants, this is not true for the rearrangement. Therefore, the mapping from $f$ to $\fast$ is not between equivalence classes modulo constants, but between individual functions. 
	
	A second issue relates to rearrangeability. When defining the decreasing rearrangement for functions in spaces like $L^p$ (see Property~{\ref{R-Lp}}) or weak-$L^p$, the rearrangeability condition is automatically satisfied. This is not true for functions in BMO, which may fail to be rearrangeable. Nevertheless, as functions in BMO are locally integrable, hence finite almost everywhere, we will have that $f\in\bmo$ is rearrangeable provided that $\mu_f(\alpha) < \infty$ for some $\alpha \geq 0$. This property is preserved when we add constants.
	
	If $f$ were not rearrangeable, then defining $\fast$ would lead to $\fast\equiv\infty$. This is the case, for instance, for $-\log|x|$, the prototypical unbounded function in $\bmo(\Rn)$. On the other hand, the positive part $(-\log|x|)_+$ is rearrangeable, as is any other $\bmo$-function of compact support, since such functions are integrable.

	
	\section{\bf Some examples}
	\label{sec:examples}
	
	\subsection{Rearrangeable functions}
	To go beyond the case of bounded functions, functions of compact 
	support, and integrable functions, we consider some examples 
	in $\bmo(\R)$ defined, pointwise, as series.
	Fix a nonnegative, nonconstant integrable function $g$
	in $\bmo(\R)$ vanishing outside $I:=(-1,1)$. Define
	\begin{equation}
		\label{eq-sum} 
		f = \displaystyle{\sum_{k = 1}^\infty g_k}\,,
	\end{equation}
	where 
	each $g_k$ is obtained from $g$ by scaling, dilation and  translation; that is,
	\begin{equation}
		\label{eq-tree}
		g_k(x) = a_k g\bigl(b_k^{-1}(x - n_k)\bigr)\,,
	\end{equation}
	for some sequences $\{a_k\}$, $\{b_k\}$, $\{n_k\}$ of positive numbers. Here $\{a_k\}$ is assumed to be bounded,
	$n_1\ge b_1$, and $n_k\uparrow \infty$. From Property~{\ref{B-trans-dil}} it follows that $\|g_k\|_{\bmo}=a_k\|g\|_{\bmo}$. Note that
	$g_k$ vanishes outside $I_k := (n_k-b_k, n_k+ b_k)$, 
	and we further assume that consecutive intervals are {\em
		well-spaced}, namely
	\begin{equation}
		\label{eq-well-spaced}n_{k + 1}-n_k \ge 9(b_k+b_{k+1}),
	\end{equation}
	so that the larger intervals $\widetilde{I}_k:= (n_k-9b_k, n_k+9b_k)$  
	are disjoint. This implies that if an  interval $J$ intersects
	$I_k$ and an adjacent interval $I_{k\pm 1}$, then 
$|J \cap \tilde I_k| \geq 8b_k\ge 4|J\cap I_k|$.

The series in Eq.~\eqref{eq-sum} converges 
	pointwise and in $\Loneloc(\R)$, and
	$$
	\int_J f \le 
\|g\|_{L^1}  \sum_{k: J\cap I_k\ne \emptyset}
	a_k b_k 
	$$
	It converges in $L^1(\R)$ if and only if the sequence $\{a_kb_k\}$
	is summable; in that case,
	$$||f||_{L^1} = \|g\|_{L^1} \sum_{k\ge 1} a_k b_k\,.$$
	We claim that $f\in\bmo$, with
	\begin{equation}
		\label{series-bound}
		\|f\|_{\BMO{}{}} = a\, \|g\|_{\BMO{}{}}\,,
	\end{equation}
	where $a:=\sup a_k$. In particular, applying this to the tail of 
	the series, we see that the convergence of the series in 
	the BMO-seminorm is equivalent to $a_k \to 0$.
	
	To prove Eq.~\eqref{series-bound},  recalling that $\|g_k\|_{\bmo}=a_k\|g\|_{\bmo}$ and the definition of $a$,   one direction reduces to showing
 	$\|f\|_{\BMO{}{}}\ge \,  \|g_k\|_{\BMO{}{}}$ for each $k$.  We fix $k$ and estimate the oscillation of $g_k$ on an interval $J$. If $J$ intersects only $I_{k}$,
	then $\cO(g_{k}, J) = \cO(f,J)$.  On the other hand, as was already pointed out, by 
	the well-spacing assumption Eq.~\eqref{eq-well-spaced}, if $J$ is sufficiently long to intersect
	 $I_k$ and one of its neighbors, it must satisfy $|J| \geq 8b_k\ge 4|J\cap I_k|$.  We now apply the calculation in Eq.~\eqref{eq-B-support2} 
to the function $g_k$ with $S = J$ and an interval $S'$ whose half consists of $J \cap I_k$, noting that $|S'| = 2|J\cap I_k| \leq 4b_k$ implies $S'$ does not intersect any neighbor of $I_k$.  Thus
	 $$\cO(g_{k}, J) \le \frac{2|S'|}{|J|}\cO(g_{k}, S') \le \cO(g_{k}, S') = \cO(f, S')$$
	 and we have shown that $\cO(g_{k}, J) \leq \|f\|_{\BMO{}{}}$.

         The other direction follows by similar arguments applied this time to
	the mean oscillation of $f$ on an interval $J$.  Again we have that 
	if $J$ intersects exactly one interval $I_{k}$,
	then $\cO(f,J)=\cO(g_{k}, J)\le a_{k}\|g\|_\bmo$.  Otherwise, by the subadditivity of oscillation and Eq.~\eqref{eq-B-support3}, noting once more that $|J| \ge 4|J\cap I_k|$ for each $k$ for which $I_k\cap J\ne \emptyset$, we write
	$$
	\cO(f,J) 
	\le \sum_{k: I_k\cap J\ne \emptyset} \cO(g_k,J)\le 
	\sum_{k: I_k\cap J\ne \emptyset} 4 \frac{|I_k\cap J|}{|J|}
	\|g_k\|_{\bmo}
	\le  \frac{a\, 
	\|g\|_\bmo}{|J|}\sum_{k: I_k\cap J\ne \emptyset} 4 |I_k\cap J|\,.
	$$
A final application of the consequences of well-spacing, namely that $\widetilde{I}_k:= (n_k-9b_k, n_k+9b_k)$ are disjoint,
	 gives
	$$
	 \sum_{k: I_k\cap J\ne \emptyset}4 |I_k \cap J|\,  \leq \sum_{k: I_k\cap J\ne \emptyset}|\widetilde I_k \cap J| \le|J|.
	$$
	
	To check whether $f$ is rearrangeable, we compute 
	its distribution function.
	By disjointness of the supports of the $\{g_k\}$ and Eq.~\eqref{eq-trans-dil},
	$$
	\mu_f(\alpha)
	= \sum_{k=1}^\infty \mu_{g_k}(\alpha)
	= \sum_{k = 1}^\infty b_k \mu_g (a_k^{-1}\alpha)
	$$
	for any $\alpha>0$.
	In particular, recalling that  
	$a:=\sup a_k$ and assuming $b:= \sum_{k\ge 1 } b_k<\infty$, we have
	$\mu_f(\alpha)
	\le b \mu_g \bigl(a^{-1}\alpha\bigr)$
	and $f^*(s)\le a g(b^{-1}s)$.
	Even when $\{b_k\}$ is not summable, $f$ may be rearrangeable 
	provided that $\{a_k\}$ decays sufficiently quickly.
	
	To understand this phenomenon, we specialize to the function
	$g(x)=(-\log{|x|})_+$, which satisfies
	$g_{I}=1$, and $\|g\|_{\bmo}=2/e$.
	Its mean oscillation is maximized on any symmetric subinterval
	of $I=(-1,1)$.
	Its distribution function and decreasing rearrangement are given by
	$$
	\mu_g(\alpha)= 2e^{-\alpha}\,,\qquad g^*(s) = 
	(-\log s+\log 2)_+\,.
	$$
	
	\begin{example}
		\label{ex-trees}
		Define $f$ by Eqs.~\eqref{eq-sum} and~\eqref{eq-tree},
		where the factors $a_k$ and $b_k$ will be further specified below.
		For $\{n_k\}$ we choose
		any sequence with $n_1\ge b_1$ and satisfying Eq.~\eqref{eq-well-spaced}.
		We consider three scenarios.
		\begin{itemize}
			\item[(a)]
			{\em $f$ is rearrangeable, but $\sum g_k$
				diverges in BMO}. \ 
			Taking $a_k = 1$ and $b_k = e^{-k}$ in Eq.~\eqref{eq-sum}, 
			we obtain for the distribution
			function and decreasing rearrangement of $f$:
			$$ 
			\mu_f(\alpha) = 2e^{-\alpha}
			\sum_{k = 1}^\infty   
			e^{-k} = 2be^{-\alpha}\,,
			\qquad f^*(s)= \bigl(-\log s + \log (2b)\bigr)_+\,,
			$$
			where $b=\frac{1}{e-1}$.  Then
			$\|f^*\|_\bmo = \|g\|_\bmo=\|f\|_\bmo$. Since
			$a_k\not\to 0$, the series from Eq.~\eqref{eq-sum}
			diverges in BMO.
			In fact, $f$ cannot be approximated in
			BMO by compactly supported functions because
			$\|f\|_{\bmo(\{|x| > R\})} \not\ra 0$ 
			as $R \ra \infty$ (see Theorem 6 in \cite{bo}).
			
			\item[(b)]  {\em $f$ is not rearrangeable, but 
				$\sum g_k$ converges in BMO}.\ 
			Taking $a_k = k^{-1/2}$ and $b_k = e^{k}$,
			the series from Eq.~\eqref{eq-sum} converges in BMO.
			However, $f$ fails to be rearrangeable:
			$$
			\mu_f(\alpha) = 2\sum_{k = 1}^\infty 
			e^{k-\sqrt{k}\alpha} = \infty\,,\qquad
			f^*\equiv +\infty\,.  
			$$
			Since the partial sums are integrable functions 
			of compact support, this demonstrates that
			rearrangeability  is not preserved under limits in BMO.  
			Note also that the series diverges in $L^1(\R)$,
			because $\{a_kb_k\}$ is not summable.
			
			\item[(c)] {\em $f$ is rearrangeable,
				$\sum g_k$ converges
				in BMO, but $\inf f^*>\inf f$.}\ 
			Taking $a_k = k^{-1}$ and $b_k = e^{k}$ yields
			$$ \mu_f(\alpha) = 2 \sum_{k=1}^{\infty} e^{-k(\alpha-1)} = 
			\frac{2}{e^{\alpha-1} -1} 
			\quad (\alpha>1)\,,\qquad
			f^*(s) = 1+ \log\bigl(1+\tfrac{2}{s}\bigr)\,.
			$$
			Although neither $f$ nor $f^*$ is integrable, 
			$f^*$ is integrable at the origin
			and $\|f^*\|_{\bmo}\le 2\|f\|_{\bmo}= \frac4e$. In this example, $\inf f=0$ but
$\inf \fast=1$.
		\end{itemize}
	\end{example}
	
	Similar series can be constructed in VMO,
	by taking $g(x)=(-\log|x|)_+^p$ with $p\in (0,1)$, and adjusting
	$a_k, b_k, n_k$ accordingly. 
	
	\subsection{Convergence of rearrangements}
	
	The next two examples show that the decreasing 
	rearrangement is discontinuous on $\bmo$ and 
	$\vmo$. We consider sequences of the
	form 
	\begin{equation}
		\label{eq-seq}
		f_k=f+g_k\,,
	\end{equation}
	where $f$ and $g$ are
	fixed functions of compact support, and $g_k$ is obtained
	from $g$ by scaling, dilation, and
	translation, as in Eq.~\eqref{eq-tree}.
	
	\begin{example} \label{ex-discont}
		{\em $f$, $f_k$ rearrangeable on a finite interval, 
			$f_k\to f$ in $\bmo$, but
			$f_k^*\not \to f^*$ in $\bmo$.}\ 
		Choose  $f=\Chi_{(0,1)}$ and $g=(-\log|2x|)_+$,
		and let $a_k=\frac1k$, $b_k=1$, and $n_k=-\frac12$,
		see Figure~2.
		Since $a_k\to 0$, the sequence $f_k$ converges to
		$f$ pointwise (except at $x=-\frac12$) and 
		in BMO. 
		
		The rearrangements are given by
		$$
		f_k^\ast(s)=
		\begin{cases}
			-\frac{1}{k}\log s\,,\quad &0\leq s < e^{-k}\\
			1\,,& e^{-k}\le s < 1+ e^{-k}\\
			-\frac{1}{k}\bigl(\log(s-1)\bigr)_+\,&s\ge 1+e^{-k}\,.
		\end{cases}
		$$
		By monotone convergence, $f_k^*\to f=f^*$ 
		pointwise (except at $x=-\frac12$) and in $L^1$.  
		However, since $f_k^*$ is constant on a short
		interval $J$ centred
		at $1$, we find  that
		$$\|f_k^\ast-\fast\|_\bmo\ge \cO(f_k^*-\fast,J) =\cO(f^*,J)=\frac12
		$$
		for all $k$.
	\end{example}
	
	Example~\ref{ex-discont} 
	can easily be modified to show that the decreasing 
	rearrangement is not continuous from 
	$\bmo(\Omega)$ to $\bmo(0,|\Omega|)$ for other domains.
	In particular, $f_k$ and $f$ can be considered
	as functions on $\R^n$ that happen to depend
	only on the first component. By scaling, translation,
	and restriction, one obtains examples on any 
	bounded domain $\Omega\subset\R^n$.
	
	\begin{figure}[t]\label{fig-jump}
	\includegraphics[width=0.8 \linewidth]{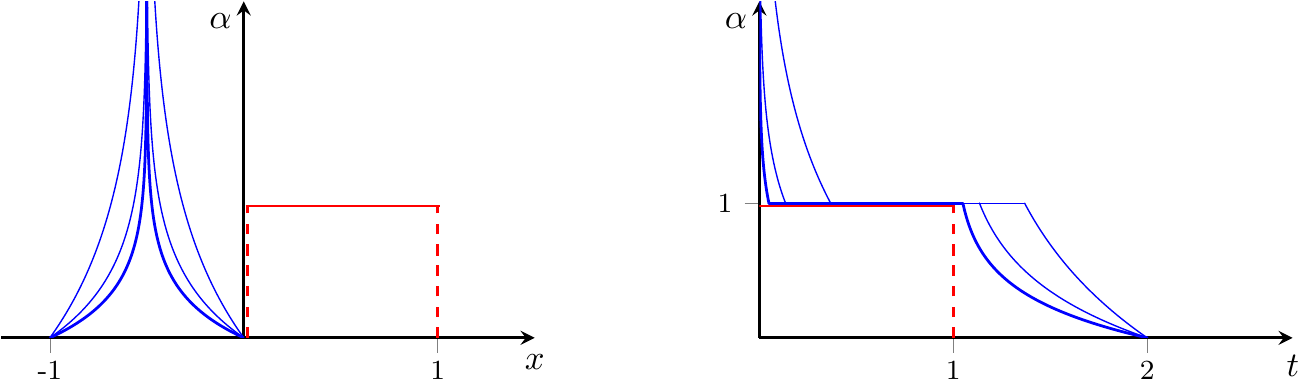}
	\caption{\small A convergent sequence $\{f_k\}$
	in $\bmo(-1,1)$ whose decreasing rearrangements $\{\fast_k\}$
do not converge in $\bmo(0,2)$.
	}
	\end{figure}
	
	If we replace the function $g(x)=(-\log|x|)_+$ 
	in Example~\ref{ex-discont}
	with $g(x)=(-\log|x|)_+^p$ for $p\in (0,1)$,
	we obtain examples of functions in 
	$\bmo\setminus\vmo$ with rearrangements in VMO:
	Since $f_k=f+g_k \not \in \vmo(\R)$ as they have
	jump discontinuities at $x=0$, while $f_k^*\in \vmo(\R_+)$.
	
	Even on the subspace VMO,
	the decreasing rearrangement is not continuous 
	unless additional conditions are imposed either on the 
	sequence of functions, or on the domain and the collection
	of shapes, see Theorems~\ref{thm-vmo-finite}
and~\ref{thm-vmo-infinite}.  
	
	\begin{example}
		\label{ex-nocont}
		{\em $f$, $f_k$ rearrangeable, $f_k\to f$ in VMO, but
			$f_k^*\not \to f^*$ in BMO.}\ 
		We again consider a sequence $f_k$
		given by Eq.~\eqref{eq-seq},
		with $g_k$ given by Eq.~\eqref{eq-tree}, but this time with 
		$$ f(x)= \sqrt{\left(-\log \bigl(\tfrac12|x+6|\bigr)\right)_+}\,,\qquad
		g(x) = \sqrt{(-\log |x|)_+}\,,
		$$
		and $a_k = k^{-\frac12}$, $b_k=e^k$, $n_k=b_k+k+1$,
		see Figure~3.
		The distribution function and decreasing rearrangement of $f$ are
		given by 
		$$
		\mu_f(\alpha)=4e^{-\alpha^2},\qquad 
		f^*(s)=\sqrt{(-\log s+\log 4)_+}\,.
		$$
		Using the properties of scaling and dilation (see 
		Property~{\ref{B-trans-dil}}), we have
		$$
		\|g_k\|_{L^1}=k^{-\frac12}e^k\|g\|_{L^1}\,,
		\qquad \|g_k\|_\bmo= k^{-\frac12}\|g\|_\bmo\, ,
		$$
		hence $f_k\to f$ in $\Loneloc(\R)$ and in $\bmo(\R)$, 
		but not in $L^1(\R)$. 
		
		Since the supports of $f$ and $g_k$
		are disjoint, the
		distribution function of $f_k$ is given by 
		$$\mu_{f_k}(\alpha)=\mu_f(\alpha)+\mu_{g_k(\alpha)}
		= \mu_f(\alpha)+ 2 e^{k(1-\alpha^2)},
		$$
		see Eq.~\eqref{eq-trans-dil}. 
		The right-hand side increases with $k$ for
		$0\le \alpha<1$, decreases for $\alpha>1$, and 
		$$
		\lim_{k\to\infty} \mu_{f_k}(\alpha) =
		\begin{cases} 
			+\infty\,, & 0\le \alpha< 1\,,\\
			4e^{-1}+2\,, & \alpha=1\,,\\
			\mu_f(\alpha)\,,\quad & \alpha>1\,.\\
		\end{cases}
		$$
		
		\begin{figure}[t]\label{fig-trees}
			\includegraphics[width=0.8 \linewidth]{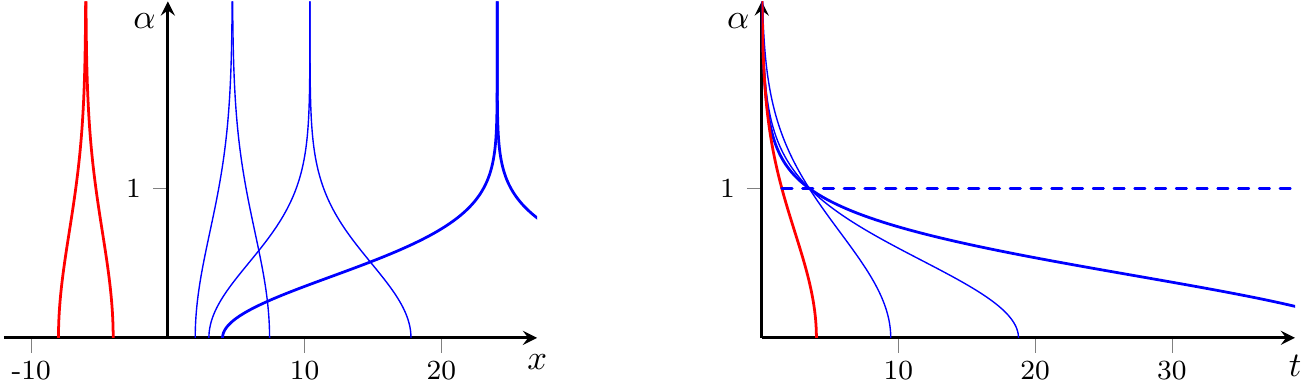}
			\caption{\small A convergent sequence
$f_k\to f$ in $\vmo(\R)$ whose decreasing rearrangements
$f_k^*$ do not converge to $\fast$.}
%
		\end{figure}
		
		We decompose $f_k= \min\{f_k,1\} + (f_k-1)_+$.
		The distribution function of the first summand is given by
		$\mu_{\min\{f_k,1\}}(\alpha) =
		\mu_f(\alpha) +2 e^{k(1-\alpha^2)}$ for $0\le \alpha <1$
		and vanishes for $\alpha\ge 1$.  Since this increases with $k$,
		Eq.~\eqref{eq-fast-alt} yields
		$$
		\lim_{k\to\infty}
		\min\{\fast_k(s),1\}
		=\Bigl| \bigcup _{k\ge1}\{\alpha\in (0,1]:
		\mu_f(\alpha) + 2e^{k(1-\alpha^2)} >s \}\Bigr|= 1\,,
		$$
		where we have used continuity of the measure
		in the first step.  
		
		The second summand has distribution function
		$ \mu_{(f_k-1)_+}(\alpha) = \mu_f(\alpha+1)
		+ 2e^{-k(2\alpha+\alpha^2)}$,
		which decreases with $k$. Eq.~\eqref{eq-fast-alt} yields
		\begin{align*}
			\lim_{k\to\infty}
			(f^*_k-1)_+(s) 
			&=\Bigl| \bigcap _{k\ge1}\{\alpha>0:
			\mu_f(\alpha+1) + 2e^{-k(2\alpha+\alpha^2)}>s \}\Bigr|\\
			&=\bigl| \{\alpha>0:\mu_f(\alpha+1)\ge s \}\bigr|\\
			&= (f^*-1)_+(s)\,.
		\end{align*}
		We have used that the level sets 
		of $f_k$ have finite measure to apply
		continuity from above. Since
		$\mu_f$ is strictly decreasing,
		the set where $\mu_f(\alpha+1)=s$ consists of a single point.
		It follows that
		$$
		f_k^*= \min\{f_k^*, 1\} + (f_k^*-1)_+
		\ \longrightarrow\ 
		\max\{f^*,1\}
		$$
		pointwise on $\R$ as $k\to\infty$. By dominated convergence,
		$$
		\lim_{k\to\infty} \|f_k^*-f^*\|_\bmo \ge 
		\lim_{k\to\infty} \cO\bigl(f_k^*-f^*,(0,2)\bigr)
		= \cO\bigl ( (1-f^*)_+,(0,2)\bigr)>0\,.
		$$ 
		This failure to converge
		is accompanied by a loss of mass reminiscent of Fatou's lemma:
		$$\|\lim \fast_k\|_\bmo =\|\max\{f^*,1\}\|_\bmo < \|f^*\|_\bmo=
		\|(\lim f_k)^*\|_\bmo\,, $$ 
		see Property~{\ref{O-trunc}}.
	\end{example}
	
	
	For the next two examples, we consider sequences of the form 
	\begin{equation}\label{form}
		f_k=f-g_k\,,
	\end{equation}
	where $f$ and $g$ are fixed functions defined on $\R_+$, and $g_k$ is obtained from $g$ in the following way:
	\begin{equation}\label{gk}
		g_k(x)=\min\left\{\tfrac{1}{k}g(x-n_k), 1\right\}
	\end{equation}
	for some choice of positive increasing sequence $\{n_k\}$. 
	
	The following example shows that on domains of
	infinite measure, the decreasing rearrangement of 
	a convergent sequence in $\vmo$ need not
	converge globally on $\R_+$.
	
	\begin{figure}[t]
		\includegraphics[width=0.8 \linewidth]{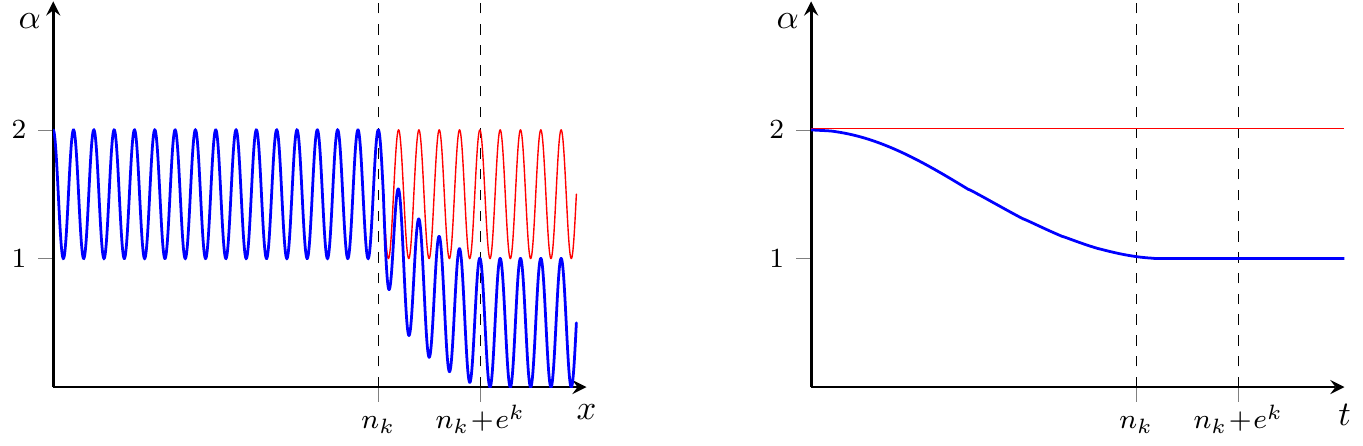}
		\caption{\small A sequence
$f_k\uparrow f$ that converges in $\vmo(\R_+)$
whose decreasing rearrangements $\fast_k$ do not converge
to $\fast$.
}
	\end{figure}
	
	\begin{example} 
		\label{ex-local}
		{\em $f$, $f_k$ rearrangeable in $\vmo(\R_+)$, $0\le f_k\uparrow f$
			in $\bmo(\R_+)$, but 
			$\fast_k\not\to \fast$ in $\BMO{}{}(\R_+)$.}\ 
		Choose $f$ to be the periodic function given by $f(x)=\frac{1}{2}\cos (\frac\pi 2 x)+\frac{3}{2}$ on $\R_+$ and $g(x)=(\ln(x))_+$ on $\R_+$, and let $n_k$ the smallest integer divisible by 4 with $n_k\ge k e^k$, see 
Figure~4. Since $g_k\le g_{k+1}$
		and $\|g_k\|_\bmo\le \frac1k \|g\|_{\bmo}\to 0$,
		the sequence $f_k$ converges to $f$
		monotonically, pointwise, and in $\bmo$.

To see that $\fast_k\not \to \fast$,
we derive a lower bound on the oscillation of
		$\fast_k-\fast$ over the interval $J=(0, n_k+e^k)$.
		The oscillation of $\fast_k$ on $J$ 
		equals that of $\max\{f_k,1 \}$ on $J$, and so 
		\[
		\begin{split}
			\cO(\fast_k,J) &= \cO(\max\{f_k,1 \},J) \\
			& \geq \frac{n_k+e^k}{n_k}\cO(\max\{f_k,1 \},(0,n_k))\\
			& \geq\frac{1}{2}\cO(\cos x,(0,\pi)) 
			= \frac{1}{\pi}.
		\end{split}
		\]
		Since $\fast\equiv 2$, it follows 
with Property~\ref{O-const} that $\|\fast_k-\fast\|_\bmo
		\ge \cO(\fast_k,J)\geq \frac{1}{\pi}$.
	\end{example}

Our final example shows that on domains of
infinite measure,
$L_k=\inf \fast_k$ need not converge to $L:=\inf\fast$,
even if $f_k\uparrow f$ pointwise and 
$(L-\fast_k)_+\rightarrow{0}$ in $\bmo(\R_+)$.
	
	\begin{example} 
		\label{ex-inf}
		{\em $f$, $f_k$ rearrangeable, $(L-\fast_k)_+\rightarrow 0$ in $\bmo(\R_+)$, but $L_k\not\rightarrow L$.}\ 
		Still considering sequences given by Eq.~\eqref{form}, with $g_k$ given by Eq.~\eqref{gk}, take $f\equiv 2$ on $\R_+$ and 
		$g(x)=(\ln(x))_+$ on $\R_+$, and let $\{n_k\}$ be any increasing sequence of positive numbers. Since $f$ and $f_k$ are continuous and decreasing, they coincide with their decreasing rearrangements. Moreover, $L=2$ and $\|(L-\fast_k)_+\|_{\bmo} = \|g_k\|_{\bmo}\leq \frac{1}{k} \|g\|_{\bmo}\to 0$. On the other hand, $L_k=0$ for every $k$, and so $L_k\not\rightarrow L$. 
	\end{example}
	
	
	\section{\bf Rearrangements on VMO} \label{sec:VMO}

	In this section, we consider 
	rearrangements on VMO and prove the main results.
	Recall, as explained in the introduction, that rearrangements are 
	fundamentally nonlinear, and so boundedness does not imply continuity. 
	We have already seen in Example~\ref{ex-discont}
	that the decreasing rearrangement fails to be continuous
	on $\bmo(\Omega)$.


\subsection{Boundedness}

	We first  show that under suitable assumptions on 
	the basis $\cS$, the decreasing rearrangement 
	of any rearrangeable function
	$f\in \VMO{\cS}{}(\Omega)$ lies in $\vmo(0,|\Omega|)$.
	In addition to the boundedness of 
	the decreasing rearrangement on $\BMO{\cS}{}$,
	this requires the following assumption on $\cS$.
	
	\noindent{\bf Density condition:}
	There exists $q\in (0,\frac14]$
	such that for every measurable $E\subset \Omega$
	with $|E||E^c|>0$, 
	\begin{equation}
		\label{eq-meas-cont}
		\limsup_{S \in \cS, |S| \ra 0} \rho(E,S)(1 - \rho(E,S)) \ge q, 
		\quad \text{where} \quad \rho(E,S) := \frac{|E\cap S|}{|S|}\,. 
	\end{equation}
	
	Note that implicit in this is the existence of shapes of arbitrarily small measure. By the Lebesgue density theorem and continuity 
of the integral, this condition holds for
the standard bases $\cB$, $\cQ$, and~$\cR$.
	
	\begin{theorem}
		\label{thm-vmobound}
		Assume that $\cS$ satisfies the density condition~\eqref{eq-meas-cont} and that $\fast\in \bmo(0,|\Omega|)$ 
		whenever $f\in\BMO{\cS}{}(\Omega)$ is rearrangeable, with 
		\begin{equation}
			\label{eq-bmobound}
			\|\fast\|_{\bmo}\leq c \|f\|_{\BMO{\cS}{}}\,.
		\end{equation}
		Then, $\fast\in \VMO{}{}(0,|\Omega|)$ 
		whenever $f\in\VMO{\cS}{}(\Omega)$. 
	\end{theorem}
	
	We need two technical lemmas. The first provides a sufficient condition for a nonnegative decreasing function of a single variable to be in VMO.
	
	
	\begin{lemma}\label{vmo-mono}
		Let $I$ be an open interval (possibly infinite) with 
		left endpoint at the origin, and let $g\in \Loneloc(I)$
		be nonnegative and decreasing.
		Then $g\in\vmo(I)$ if and only if 
		$g$ is continuous on $I$ and
		\begin{equation}
			\label{eq-vmo-mono}
			\lim_{\delta \ra 0} \sup_{J \subset [0,\delta) \cap I}\cO(g,J) =0.
		\end{equation}
	\end{lemma}
	
	\begin{proof} 
		If $g\in\vmo(I)$, then Eq.~\eqref{eq-vmo-mono} holds by definition. Furthermore, $g$ cannot have jump discontinuities hence, being monotone, is continuous on $I$.
		
		Conversely, suppose that $g$ is continuous and Eq.~\eqref{eq-vmo-mono} holds.
		Given $\eps > 0$, take $\delta > 0$ so that $[0,\delta) \subset I$ and $\sup_{J \subset [0,\delta)}\cO(g,J) < \eps$. 
		
		Since $g$ is continuous,
		decreasing, and bounded below, it is 
		uniformly continuous on $[\delta/2, \infty) \cap I$.  Thus 
		there exists $\eta > 0$ such that
		$|g(x)-g(y)|<\eps$ for every pair of points
		$x,y\in I$ with $x,y\ge \delta/2$ and $|x-y|<\eta$.
		By Property~\ref{O-2int}, this implies that
		$\cO(g, J) <\eps$ for every interval $J\subset [\delta/2,\infty) \cap I$ with
		$|J|<\eta$.
		
		Thus for an interval $J \subset I$ with $|J| < \min\{\eta, \delta/2\}$, either $J \subset [0,\delta)$ or $J\subset [\delta/2,\infty) \cap I$, hence $\cO(g,J) < \epsilon$.
	\end{proof}
	
	The next lemma shows
	that if $\fast$ has a jump discontinuity then the 
	oscillation of $f$ must be large on shapes of arbitrarily small measure, 
	and this can be quantified in terms of the size of the jump.  
	
	\begin{lemma}
		\label{jump}
		Suppose $\cS$ satisfies the density assumption~\eqref{eq-meas-cont}.
		Then the decreasing rearrangement
		$f^*$ of any function $f\in \VMO{\cS}{}(\Omega)$
		is continuous on $(0,|\Omega|)$.
	\end{lemma}
	
	\begin{proof}
		Let $f\in \BMO{\cS}{}(\Omega)$ be a
		rearrangeable function.  
		Replacing $f$ with $|f|$, we 
		assume without loss of generality that $f$ is nonnegative.
		
		Since $f^*$ is monotone decreasing and right-continuous, its only possible discontinuities are jumps of the form 
		$$
		\beta:=\lim_{s\rightarrow t^-}\fast(s)>f^*(t)=:\alpha\,,
		$$
		at some $t\in (0,|\Omega|)$.  We will estimate the size of the jump, $\beta - \alpha$, in terms of the modulus of oscillation of $f$.
		
		Consider the truncation $\widetilde f= \min(\max(f,\alpha),\beta)$.
		By Property~\ref{R-trunc}, 
		$$
		(\widetilde f)^*= \min(\max(f^*,\alpha),\beta)
		= \alpha + (\beta-\alpha)\Chi_{(0,t)}\,.
		$$
		This implies, since 
		$\widetilde f\ge 0$, that $\widetilde f$ agrees with 
		$\alpha+ (\beta-\alpha)\Chi_{E}$
		almost everywhere, where $E:=E_\gamma(f)$ is the level set of $f$ at 
		any $\gamma \in (\alpha,\beta)$. 
		By equimeasurability (see Property~\ref{R-equi}),
		$|E| = t>0$ and $|E^c|>0$.
		
		Given $\delta > 0$, the density assumption in Eq.~\eqref{eq-meas-cont}
		gives us the existence of a shape $S\in\cS$ with $|S|\le \delta$
		such that $\rho(E,S)(1 - \rho(E,S)) \ge q > 0$. We estimate
		$$
		\cO(f, S) \ge \cO(\widetilde f, S)
		= (\beta-\alpha) \cO(\Chi_{E}, S) \ge 2q(\beta-\alpha)\, ,
		$$
		where we have used Property~\ref{O-trunc} in the middle step,
		then applied Property~\ref{O-const}, and finally Eq.~\eqref{eq-char-func}. As $\delta>0$ was arbitrary, we have
		$$
		\beta-\alpha  \le (2q)^{-1}\limsup_{|S| \ra 0} \cO(f, S)\, .
		$$
		Thus for 
		$f\in \VMO{\cS}{}(\Omega)$, $f^*$ cannot have any jumps.
	\end{proof}
	
	
	\begin{proof}[Proof of Theorem~\ref{thm-vmobound}]
		Suppose $f\in\VMO{\cS}{}(\Omega)$ is rearrangeable. 
		We need to show $f^\ast\in \vmo(0,|\Omega|)$.
		By Property~\ref{O-abs}, it suffices to consider 
		nonnegative $f$, and by Lemma~\ref{jump} we may further
		assume that $f^*$ is continuous.
		
		Therefore, by Lemma \ref{vmo-mono},
		we only need to show that $f^*$ has vanishing mean oscillation at the origin.  To do so, we will bound
		$\cO(\fast,J)$
		for  $J \subset [0,\delta)$ with $\delta < |\Omega|$.

		Writing $J = (a,b)$ and $\beta=\fast(b)$, consider the 
		function $g=(\fast-\beta)_+ = (f-\beta)_+^*$. 
		Since $\fast \in \bmo(0,|\Omega|)$,
		and $\fast \geq \beta$ on $J$, we have by 
		Eq.~\eqref{eq-bmobound} that
		$$
		\cO(\fast, J)=\cO((\fast-\beta)_+, J)
		\le \|(\fast-\beta)_+\|_\bmo\le c\|(f-\beta)_+\|_\bmo\,.
		$$
		Let $S$ be a shape with $\cO((f-\beta)_+,S)\ge
		\frac12 \|(f-\beta)_+\|_\bmo$.
		It follows from the above, together with 
		Property~\ref{O-median} and the equimeasurability of $(f-\beta)_+$ and $g$, that
		$$
		\cO(\fast, J)\le 2c \, \cO((f-\beta)_+,S) \le \frac{4c}{|S|}
		\int_\Omega (f-\beta)_+ = \frac{4c}{|S|} \int_0^{b} g \leq \frac{4c}{|S|}
		\int_0^{\delta}\fast\,.
		$$
		For any $\eta > 0$, we therefore get, using Property~\ref{O-trunc}, that
		$$
		\cO(\fast, J) \leq \max\left\{2c\sup_{|S| < \eta}\cO(f,S),  \frac{4c}{\eta}\int_0^{\delta}\fast\right\}.
		$$
		Since $f\in\VMO{\cS}{}(\Omega)$ and $f^*$ is \textcolor{magenta}{in $\bmo(0,|\Omega|)$, hence integrable at the origin (see Remark~\ref{integrability})}, we can choose $\eta$ and then $\delta$ to make the right-hand-side arbitrary small.
	\end{proof}
	
	As as consequence of this bound on the decreasing rearrangement, we are able to obtain an analogous result for the symmetric decreasing rearrangement defined by Eq.~\eqref{symmrea}.
	
	\begin{corollary}\label{corsymbound}
		If $f\in \vmo(\R^n)$ is rearrangeable then $Sf\in \vmo(\R^n)$.
	\end{corollary}
	
	To prove this corollary, we make use of the following technical lemma from \cite{BDG11} that allows for the transfer of mean oscillation estimates for the decreasing rearrangement to the symmetric decreasing rearrangement. 
			
			\begin{lemma}[\cite{BDG11}]
				\label{SDRlocalequivalent}
				Let $R > 0$ and $Q\subset B(0,R)$
				be a cube of diameter $d$, centred at a 
				point $x$ with $|x|\le R-d/2$.
				There is an interval $I\subset (0,\omega_n R^n)$ of length
				$|I|\le n \omega_nR^{n-1}d$, 
				such that if $f_1,f_2$ are rearrangeable, then 
				$$
				\cO(Sf_1-Sf_2,Q)\leq n^{\frac{n}{2}}\omega_n\,\cO(f_1^*-f_2^*,I)\,.
				$$
			\end{lemma}

	\begin{proof}[Proof of Corollary 4.4]
		Let $f\in \vmo(\R^n)$ be rearrangeable.  By the boundedness of the decreasing rearrangement on $\bmo(\R^n)$ and Theorem~\ref{thm-vmobound}, $\fast \in \vmo(\R_+)$.  Since by
		definition $Sf(x)=\fast(\omega_n |x|^n)$ , it is continuous on $\R^n\setminus\{0\}$ by 
		Lemma~\ref{jump}. Moreover, as a radially decreasing,
		nonnegative function, it is uniformly continuous
		on the complement of any centred ball of finite radius. By the same argument as in 
		Lemma~\ref{vmo-mono}, it therefore suffices to show that the modulus of oscillation on $B(0,R)$ vanishes as $R \ra 0$. 
		
		For a cube $Q$ contained in $B(0,R)$, 
		Lemma~\ref{SDRlocalequivalent},
		applied to $f_1=f$ and $f_2=0$, yields that
		$$
		\cO(Sf, Q)\le n^{\frac{n}{2}}\omega_n\cO(\fast,I)
		$$
		for an interval $I$ with $|I| \le c_n R^n$. Since $\fast \in \vmo(\R_+)$, we have
		$$\lim_{R \ra 0} \sup_{Q \subset B(0,R)} \cO(Sf,Q)=0\, .$$
\\[-1cm]
	\end{proof}

	\subsection{Continuity}
	
	In this section, we derive conditions
	on a sequence of functions $f_k$ in BMO 
      converging to a function $f$ in VMO that ensure that 
the sequence of rearrangements $\fast_k$ converges in BMO to $\fast$ in VMO. 
	
	
	
	For $\vmo$, there exists an analogue of the Arzel\`{a}-Ascoli theorem that can be used to characterize relative compactness \cite{bn1}. In our case, we take advantage of the monotonicity of rearrangements and make use of a theorem of P\'{o}lya (see \cite{po} and \cite[page 270]{ps}), given here under slightly weakened assumptions.
	
	\begin{lemma}\label{supedPolya}
		Let $f_k$, $k\in\N$, be monotone decreasing functions on $(0,b)$ for some $0<b\leq\infty$ converging almost everywhere to a continuous function $f$. Then, the convergence is uniform on any compact subinterval of $(0,b)$. Furthermore, if $b=\infty$, $f_k$, $k\in\N$, and $f$ are bounded below and $\inf_t f_k(t)\rightarrow \inf_t f(t)$, then the convergence is uniform on $[a,\infty)$ for any $a>0$. 
	\end{lemma}
	
	\begin{proof}
		Given $\eps>0$, select a partition $a<x_0<\ldots<x_n<b$ such that for each $i=1,\ldots,n$, $|f(y)-f(z)|<\eps/2$ for all $y,z\in[x_{i-1},x_i]$, and there exists $K_i$ such that  $|f_k(x_i)-f(x_i)|<\eps/2$ whenever $k\geq K_i$. Fix $x\in[x_0,x_n]$ and select $i$ such that $x\in[x_{i-1},x_i]$. Then for $k\geq\max\limits_i K_i$,
		$$
		f(x)-\eps< f(x_i)-\eps/2 \leq f_k(x_i) \leq f_k(x)\leq f_k(x_{i-1})<f(x_{i-1})+\eps/2<f(x)+\eps.
		$$ 
		Thus, for a compact subinterval $I\subset(0,b)$, if $\{f_k\}$ converges at the endpoints of $I$ this shows that $\{f_k\}$ converges uniformly to $f$ on $I$. If $\{f_k\}$ does not converge at either of the endpoints, $I$ can always be extended to a larger compact subinterval $\widetilde{I}\subset(0,b)$ on which the convergence is uniform, implying uniform convergence on $I$. 
		
		In the case $b=\infty$, then the assumption that $\inf_t f_k(t)\rightarrow \inf_t f(t)$ means that one may choose $x_n=\infty$ in the previous argument, giving the result.
	\end{proof}

	The next lemma provides a sufficient condition
	for the decreasing rearrangements of 
	a convergent sequence in $\vmo$ to be relatively compact.
	
	\begin{lemma}
		\label{lem-uniformk2}
		Let $\cS$ be a basis of shapes in a domain
		$\Omega\subset\R^n$, satisfying the
		hypotheses of Theorem~\ref{thm-vmobound}.
		Let $f_k$, $k\in\N$, and $f$ be rearrangeable
		functions in $\BMO{\cS}{}(\Omega)$.
		
		If $f \in \VMO{\cS}{}(\Omega)$, 
		$f_k \ra f$ in $\BMO{\cS}{}(\Omega)$, and
		$f_k^* \ra f^*$
		in $\Lone{(0,b)}$ for some $0<b\leq |\Omega|$, then
		$f_k^*\ra\fast$ in $\bmo(0,b)$.
	\end{lemma}
	
	\begin{proof} Fix $b \in (0, |\Omega|]$ such that $\{\fast_k\}$ converges to $\fast$ in $\Lone{(0,b)}$. Any subsequence of $\{\fast_k\}$ will then also converge to $\fast$ in $\Lone{(0,b)}$ and so have a further subsequence that converges pointwise almost everywhere to $\fast$ on $(0,b)$. It suffices to show that this subsequence, which we continue to denote by $\{f^*_k\}$, converges to $f^*$ in $\bmo(0,b)$.

		%
		
		
		As $\fast$ is continuous by Lemma \ref{jump}, and the  functions $f_k^*$ 
		are monotone decreasing and converge pointwise almost everywhere to $\fast$, 
		Lemma \ref{supedPolya} tells us that the convergence is uniform on $[\delta,b)$ for any $\delta$, $0 < \delta < b$. Note that if $b=\infty$, then the fact that $f_k^*, f^*$ are in $L^1(\R_+)$ means that $\inf_t f^*_k(t)=0= \inf_t f^*(t)$ for all $k$.  
		
		
		
		
		Given such a $\delta$, if $J \subset (0,\delta)$,  
then we have, as in the proof of Theorem~\ref{thm-vmobound}, 
		$$
		\cO(\fast, J) \leq \max\left\{2c\sup_{|S| < \eta}\cO(f,S),  \frac{4c}{\eta}\int_0^{\delta}\fast\right\}\,,
		$$
		and correspondingly for each $\fast_k$.
		Given $\epsilon > 0$, since $f \in \VMO{\cS}{}(\Omega)$, we can choose $\eta > 0$ to make  $2c\sup_{|S| < \eta}\cO(f,S) < \epsilon/2$, and the convergence of $f_k$ to $f$ in $\BMO{\cS}{}(\Omega)$ means that for this $\eta$ and all sufficiently large $k$, $2c\sup_{|S| < \eta}\cO(f_k,S) < \epsilon$.  
For this $\eta$, we can choose $\delta>0$ 
such that $\frac{4c}{\eta}\int_0^{\delta} \fast < \epsilon$, and also, 
since $\{\fast_k\}$ is convergent in $\Lone{(0,b)}$, hence uniformly 
integrable, $\frac{4c}{\eta}\int_0^{\delta}f_k^\ast < \epsilon$ for all $k$.  
		
		Combining, we get that for 
$\delta$ sufficiently small and $k$ sufficiently large, 
		$\cO(f_k^\ast - \fast, J) < 2 \epsilon$ for $J \subset (0,\delta)$.
		By uniform convergence, 
the stronger estimate $\sup_J|f_k^*-f^*|<2\eps$ 
holds when $J \subset [\delta/2,b)$.
		
If $J \subset (0,b)$ is not in one of these cases, 
then $J \supset (\delta/2,\delta)$.
		Let $g = f_k^\ast - \fast$, $I = J \cap (0,\delta)$,  $I' = J \cap (\delta/2,\delta)$. Noting that $|I'| \geq |I|/2$, we can estimate
		$$\cO(g, J) \leq 2 \fint_J |g - g_{I'}| \leq \frac{1}{|J|} \int_{I} |g - g_I| + |g_I - g_{I'}| + \frac{1}{|J|} \int_{J \setminus I} |g - g_{I'}|
		\leq 3\cO(g, I) + 2\sup_{[\delta/2, b)} |g|.$$
		Thus we have shown that for $k$ sufficiently large, $\|f_k^\ast - \fast\|_{\bmo(0,b)} \leq 10 \epsilon.$
	\end{proof}
	
	\begin{theorem} \label{thm-vmo-finite}
		Let $\cS$ be a basis of shapes in a domain
		$\Omega\subset\R^n$, satisfying the
		hypotheses of Theorem~\ref{thm-vmobound}. Let $f_k$, $k \in \N$, and $f$ be rearrangeable functions in $\BMO{\cS}{}(\Omega)$.
		
		If $f \in \VMO{\cS}{}(\Omega)$, and $f_k \ra f$ in ${\BMO{\cS}{}}(\Omega)$
		and in $L^1(\Omega)$, then $\fast_k\to \fast$ in $\bmo(0,|\Omega|)$.
	\end{theorem}
	
	\begin{proof}
		By Property \ref{R-Lp}, $\{f_k\}$ to $f$ in $L^1(\Omega)$ implies the convergence of $\{\fast_k\}$ to $\fast$ in $L^1(0,|\Omega|)$. The convergence in $\bmo(0,|\Omega|)$ follows then by taking $b=|\Omega|$ in Lemma \ref{lem-uniformk2}.
	\end{proof}
	
	We are now ready to prove the results given in the introduction. Note that in the case of $\Omega=Q_0$, the assumption of $L^1$ convergence of $\{f_k\}$ follows from $\bmo$ convergence upon normalization of the means --- see Property~{\ref{B-L1}}. 
	

\begin{proof}[Proof of Theorems~\ref{introbound} and~\ref{introcont}]
	The basis $\cQ$ is well known to satisfy the hypotheses of 
Theorem~\ref{thm-vmobound}, see the remark after the density 
condition~\eqref{eq-meas-cont}. 

Let $Q_0$ be a finite cube and $f\in \vmo(Q_0)$.
Since $Q_0$ has finite measure, $f$ is rearrangeable,
and Theorem~\ref{thm-vmobound} yields that $\fast\in\vmo(Q_0)$.
If, moreover, $f_k\to f$ in $\bmo(Q_0)$ and $\fint_{Q_0} f_k\to
\fint_{Q_0} f$, then $f_k\to f$ also in $L^1(Q_0)$.
It follows from Theorem~\ref{thm-vmo-finite} that $\fast_k\to
\fast$ in $\bmo(Q_0)$.
\end{proof}

	It remains to consider the case of infinite domains. While under the condition of the previous theorem, we have convergence of the rearrangements in $\BMO{}{}$ on any finite interval, Example \ref{ex-local} shows that convergence on all of $\R_+$ requires further assumptions at infinity.
	
	\begin{lemma}
		\label{lem-uniformk}
		Let $\cS$ be a basis of shapes in a domain
		$\Omega\subset\R^n$ of infinite measure, satisfying the
		hypotheses of Theorem~\ref{thm-vmobound}.
		Let $f_k$, $k\in\N$, and $f$ be rearrangeable
		functions in $\BMO{\cS}{}(\Omega)$, and write $L:=\inf \fast$.
		
		If $f \in \VMO{\cS}{}(\Omega)$, 
		$f_k \ra f$ in $\BMO{\cS}{}(\Omega)$,
		$f_k^* \ra f^*$
		in $\Lone{(0,b)}$ for every $0<b < \infty$, and
		$$
		\|(L-\fast_k)_+\|_{\BMO{}{}}\ra 0\quad\text{as}\quad k\rightarrow\infty,
		$$
		then
		$f_k^*\ra\fast$ in $\bmo(\R_+)$.
	\end{lemma}
The hypothesis that $(L-\fast_k)_+\to 0$ 
can be replaced by the convenient assumption
that $\inf \fast_k \to L$ as $k\to \infty$,
i.e., $\mu_{f_k}(\alpha)\to \infty$
for all $\alpha<L$. However, this 
assumption is strictly stronger, 
see Example~\ref{ex-inf}.

	\begin{proof}[Proof of Lemma~\ref{lem-uniformk}]
		Consider $f_k$, $k \in \N$, and $f$ satisfying the 
hypotheses. By Theorem~\ref{thm-vmobound},
		$\fast\in \vmo(0,|\Omega|)$. 

Exhausting $\R_+$ by intervals of 
the form $(0,b)$ for $0<b<\infty$, and using the convergence 
of $\{\fast_k\}$ to $\fast$ in $\Lone{(0,b)}$, we see that 
any subsequence of $\{\fast_k\}$ has a further subsequence that converges pointwise almost everywhere to $\fast$ on $\R_+$. Denote by $E\subset\R_+$ the set on which convergence holds. It suffices to show that this subsequence, which we continue to denote by $\{f^*_k\}$, converges to $f^*$ in $\bmo(\R_+)$.
We will show that for any $a\in E$,
\begin{equation}
\label{eq:uniformk}
 \limsup_{k\to\infty}\|\fast-\fast_k\|_{\bmo(\R_+)}\leq 4(\fast(a)-L)\,.
\end{equation}
The result follows by taking $a\rightarrow\infty$ 
from the monotonicity 
of $\fast$ and the definition of $L$.

Note that for any $0<\lambda<\infty$, we can write $\fast_k=\max\{\fast_k,\lambda\}-(\lambda-\fast_k)_+$. Taking $\lambda=L$, we have that
                \begin{equation}
                        \label{eq-vmo-split}
                        \|\fast-\fast_k\|_{\bmo(\R_+)}
                        \le \|\fast-\max\{\fast_k,L\}\|_{\bmo(\R_+)}
                        + \|(L-\fast_k)_+\|_{\bmo(\R_+)}\,.
                \end{equation}
                By assumption, the last term converges to zero as $k\ra \infty$.

For the first term on the right hand side of Eq.~\eqref{eq-vmo-split}, 
let $a\in E$, take $b\ge 2a$, and consider an arbitrary
interval $J\subset\R_+$.
For $J\subset (0,b)$ we use that 
$\fast_k(s)-\max\{\fast_k(s), L\}=(\fast_k(s)-L)_+$
decreases with $s$ to estimate
                \[
                \begin{split}
                        \sup_{J\subset(0,b)}\cO(\fast-\max\{\fast_k,L\},J)
                        &\leq \sup_{J\subset(0,b)}\cO(\fast-\fast_k,J)+\sup_{J\subset(0,b)}\cO(\fast_k-\max\{\fast_k,L\},J)\\
                        &\leq \|\fast-\fast_k\|_{\bmo(0,b)} + (L-\fast_k(a))_+\,.
                \end{split}
                \]
              As $k\rightarrow\infty$, the first term converges to zero by Lemma \ref{lem-uniformk2} and the second converges to $(L-\fast(a))_+=0$ 
since $a\in E$.
For $J\subset(a,\infty)$ we have
                \[
                \begin{split}
\sup_{J\subset(a,\infty)}
                        \cO(\fast-\max\{\fast_k,L\},J)
                        &\leq 
\sup_{J\subset(a,\infty)}\frac{1}{|J|}
\int_J |\fast-L|+
\sup_{J\subset(a,\infty)} 
\int_J |L-\max\{\fast_k,L\}| \\
                        &\leq \sup_{s\geq a}|\fast(s)-L|
                     +\sup_{s\geq a}(\fast_k(s)-L)_+\\
                        &= (\fast(a)-L)+(\fast_k(a)-L)_+\,,
                \end{split}
                \]
                which converges to $2(\fast(a)-L)$ as $k\rightarrow\infty$
since $a\in E$.
Finally, if $J\supset(a,b)$, then
                $$
                \frac{1}{|J|}\int_{J\cap (0,b)}\!|\fast-\max\{\fast_k,L\}|\leq \frac{1}{b-a}\|\fast-\max\{\fast_k,L\}\|_{\Lone(0,b)}\leq \frac{1}{a}\|\fast-\fast_k\|_{\Lone(0,b)}
                $$
                and
                \[
                \begin{split}
                \frac{1}{|J|}\int_{J\cap (b,\infty)}\!|\fast-\max\{\fast_k,L\}|
&\leq (\fast(a)-L) + (\fast_k(a)-L)_+,
\end{split}
\]
where we have again used monotonicity of $\fast$ and $\fast_k$
in the last step.
It follows that
                $$
\sup_{J\supset (a,b)}
                \cO(\fast-\max\{\fast_k,L\},J)\leq \frac{2}{a}\|\fast-\fast_k\|_{\Lone(0,b)} +2 (\fast(a)-L) + 2(\fast_k(a)-L)_+.
                $$
                As $k\rightarrow\infty$, 
the first term on the right hand side
converges to zero by assumption, and the last
term converges to $2(\fast(a)-L)$ since $a\in E$.
This completes the proof of Eq.~\eqref{eq:uniformk}.
	\end{proof}

	\begin{theorem} \label{thm-vmo-infinite}
		Let $\cS$ be a basis of shapes in a domain
		$\Omega\subset\R^n$ of infinite measure, satisfying the
		hypotheses of Theorem~\ref{thm-vmobound}.
		Let $f_k$, $k \in \N$, and $f$ be rearrangeable functions
		in $\BMO{\cS}{}(\Omega)$, and write $L:=\inf \fast$.
		
		If $f\in\vmo_{\cS}(\Omega)$, $f_k \ra f$ in ${\BMO{\cS}{}}(\Omega)$,
		$0\le f_k\uparrow f$ pointwise, and $(L-\fast_k)_+ \to 0$ in $\bmo(\R_+)$,
		then $\fast_k\to \fast$ in $\bmo(\R_+)$.
	\end{theorem}

\begin{proof}
	By Property~\ref{R-ptwise}, if  $f_k\uparrow f$ on $\Omega$, 
	then $\fast_k\uparrow\fast$ 
	on $\R_+$. By monotone convergence, $\fast_k\ra\fast$ in $\Lone(0,b)$ for any $b<\infty$, and we can apply Lemma \ref{lem-uniformk}.
\end{proof}

By Lemma~\ref{SDRlocalequivalent}, the conclusion of 
Theorem~\ref{thm-vmo-infinite} directly  extends to the symmetric decreasing
	rearrangement.
	
	\begin{corollary}\label{corsymcont} Under the hypotheses of 
		Theorem~\ref{thm-vmo-infinite}, $Sf_k\rightarrow Sf$ in $\bmo(\R^n)$.
	\end{corollary}

\section*{{Acknowledgements}}

Thanks to Ruben Calzadilla-Badra for discussions on P\'{o}lya's theorem in
connection with an NSERC Undergraduate Student Research Award
project at Concordia. 


\end{document}